\documentclass{article}
\usepackage{graphicx}
\usepackage{amssymb, amsmath, amsthm}
\usepackage{bbm, multicol, multirow, hhline}
\usepackage[round]{natbib}

\usepackage{tikz}
\usetikzlibrary{arrows,shapes.arrows,shapes.geometric}

\usepackage{amsmath, amsfonts, amsthm, amssymb, bbm, bm}
\usepackage{mathrsfs}
\usepackage{enumerate}
\usepackage{tabularx}
\usepackage{multicol}
\usepackage{multirow}
\newcommand{\Ind}{\mathbbm{1}}

\newcommand{\reals}{\mathbb{R}}

\newcommand\indep{\protect\mathpalette{\protect\independenT}{\perp}}
\def\independenT#1#2{\mathrel{\rlap{$#1#2$}\mkern2mu{#1#2}}}

\theoremstyle{plain}
\newtheorem{lem}{Lemma}[section]
\newtheorem{thm}[lem]{Theorem}
\newtheorem{prop}[lem]{Proposition}
\newtheorem{cor}[lem]{Corollary}
\newtheorem{conj}[lem]{Conjecture}
\theoremstyle{definition}
\newtheorem{dfn}[lem]{Definition}
\newtheorem{rmk}[lem]{Remark}
\newtheorem{exm}[lem]{Example}

\newcommand{\benum}{\begin{enumerate}}
\newcommand{\eenum}{\end{enumerate}}

\newcommand{\bitem}{\begin{itemize}}
\newcommand{\eitem}{\end{itemize}}

\newcommand{\barr}{\begin{array}}
\newcommand{\earr}{\end{array}}

\newcommand{\bmat}{\begin{pmatrix}}
\newcommand{\emat}{\end{pmatrix}}

\newcommand{\blist}{\renewcommand{\labelenumi}{\textbf{\arabic{enumi}}.} \begin{enumerate}}
\newcommand{\elist}{\end{enumerate} \renewcommand{\labelenumi}{\arabic{enumi}.}}

\newcommand{\bs}{\boldsymbol}

\def\bal#1\eal{\begin{align*}#1\end{align*}}

\title{Smoothness of marginal log-linear parameterizations}
\author{Robin Evans}
\newcommand{\bmllp}{
\begin{tabular}{r|l}
$M_i$&$L$\\
\hline}
\newcommand{\emllp}{\end{tabular}}

\newcommand{\nud}{\mspace{-2mu}}
\newcommand{\sn}{\scriptscriptstyle\nud}

\begin{document}

\maketitle

\begin{abstract}
  % Many multivariate statistical models are defined by restrictions on
  % certain marginal or conditional distributions, motivating
  % parameterizations which enable such constraints to be enforced
  % easily.  If such a parameterization is shown to be smooth then it
  % follows that the model itself is a curved exponential family of
  % distributions.  In the discrete case arbitrary conditional
  % independence models can be defined by such restrictions on the
  % marginal log-lienar parameters of \citet{br:02}; however in some
  % cases it is unknown whether the resulting parameterizations are
  % smooth, and therefore whether the models are regular.
  
  We provide results demonstrating the smoothness of some marginal
  log-linear parameterizations for distributions on multi-way
  contingency tables.  First we give an analytical relationship
  between log-linear parameters defined within different margins, and
  use this to prove that some parameterizations are equivalent to ones
  already known to be smooth.  Second we construct an iterative method
  for recovering joint probability distributions from marginal
  log-linear pieces, and prove its correctness in particular cases.
  Finally we use Markov chain theory to prove that certain cyclic
  conditional parameterizations are also smooth.  These results are
  applied to show that certain conditional independence models are
  curved exponential families.
\end{abstract}

\section{Introduction}

Models for multi-way contingency tables may include restrictions on
various marginal or conditional distributions, especially in the
context of longitudinal or causal models \citep[see, for example,][and
references therein]{lang:94, bergsma:09, evans:13}.  Such models can
often be parameterized by combining log-linear parameters from within
different marginal tables.  The resulting \emph{marginal log-linear}
parameterizations, introduced by \citet{br:02}, provide an elegant and
flexible way to parameterize a multivariate discrete probability
distribution.

Setting these marginal log-linear parameters to zero can be used to
define arbitrary conditional independence models \citep{rudas:10,
  forcina:10}, including those corresponding to undirected graphical
models or Bayesian networks.  If these zero parameters can be embedded
into a larger smooth parameterization of the joint distribution, then
the model defined by the conditional independence constraints is a
curved exponential family, and therefore possesses good statistical
properties.  This approach is applied by \citet{rudas:10} and
\citet{evans:13} to classes of graphical models.

% Log-linear parameters provide an attractive, flexible, and variation
% independent way of parameterizing a multivariate probability
% distribution.  Setting certain log-linear parameters to zero can be
% used to define several important classes of statistical model,
% including undirected graphical models and hierarchical models.
% \citet{br:02} study log-linear parameters defined within particular
% marginal distributions; this leads to much greater flexibility in
% model specification: for example, setting such parameters to zero can
% be used to define arbitrary conditional independence models
% \citep{rudas:10, forcina:10}.  If those same parameters can be
% embedded into a larger smooth parameterization of the joint
% distribution, then the model defined by the conditional independence
% constraints is a curved exponential family, and therefore possesses
% good statistical properties.  This approach is applied by
% \citet{rudas:10} and \citet{evans:13} to certain classes of graphical
% models.

Unfortunately, there exist models of conditional independence
which---though believed to be curved exponential families---cannot be
embedded into parameterizations currently known to be smooth.
\citet{forcina:12} studies examples of models defined by `loops' of
conditional independences, such as
\begin{align}
X_1 \indep X_2 \,|\, X_3, \qquad X_1 \indep X_3 \,|\, X_4, 
\qquad X_1 \indep X_4 \,|\, X_2, \label{eqn:cyclemod}
\end{align} 
which can be defined by constraints on the conditional distributions
$p_{2|13}$, $p_{3|14}$ and $p_{4|12}$ respectively.  However it is not
clear whether a smooth parameterization of the joint distribution can
be constructed using these conditionals.  The model can also be
defined by setting a particular collection of marginal log-linear
parameters to zero (see Section \ref{sec:cycles} for details), but
there is no way to embed these parameters into a smooth
parameterization of the kind studied by \citet{br:02}, so their
results do not apply.  \citet{forcina:12} gives a numerical test for
this model which is highly suggestive of smoothness, but no formal
proof is available.

The contribution of this paper is to show that the class of smooth
discrete parameterizations which can be constructed using marginal
log-linear (MLL) parameters is considerably larger than had previously
been known, and that models such as (\ref{eqn:cyclemod}) can indeed be
embedded into these parameterizations.  We give three different
methods for demonstrating smoothness in this context.  First we
provide an analytical expression for the relationship between
log-linear parameters defined within different marginal distributions;
this allows us to prove the equivalence of various parameterizations.
Second we show that particular fixed point maps relating different
parameters are contractions, and hence can be used to uniquely recover
the joint probability distribution.  Lastly we use Markov chain theory
to show that we can smoothly recover joint probability distributions
from `cyclic' conditional distributions; this is used to show that
certain conditional independence models, including the one above, are
curved exponential families of distributions.

The rest of the paper is organized as follows: Section \ref{sec:mllps}
reviews marginal log-linear parameters and their properties.
% , and
% demonstrates how they may be used to parameterize conditional
% distributions.  
Section \ref{sec:smooth} specifies the relationship between log-linear
parameters defined within different margins, enabling certain
parameterizations to be proven equivalent.  Section \ref{sec:fixed}
extends this by constructing fixed point methods that smoothly recover
a joint distribution.  Section \ref{sec:cycles} further extends the
results of Section \ref{sec:smooth} using Markov chain theory, and
demonstrates that certain conditional independence models are curved
exponential families.  Section \ref{sec:discuss} contains 
discussion, and a conjecture on the precise characterization of smooth
MLL parameterizations.

\section{Marginal Log-Linear Parameters} \label{sec:mllps}

We consider multivariate distributions over a finite collection of
binary random variables $X_v \in \{0,1\}$, for $v \in V$; we denote
their joint distribution by $p_V \equiv p(x_V) \equiv P(X_V = x_V)$.
All the results herein also hold (or have analogues) in the case of
general finite discrete variables, but the notation becomes more
cumbersome.  For $M \subseteq V$ we denote the marginal distribution
over $X_M = (X_v,\, v \in M)$ by $p_M \equiv p(x_M) \equiv P(X_M =
x_M)$, and for disjoint $A,B \subseteq V$ we denote the relevant
conditional distribution by $p_{A|B} \equiv p(x_A \,|\, x_B) \equiv
P(X_A = x_A \,|\, X_B = x_B)$.  Distributions are assumed to be
strictly positive: $p_V > 0$.

\begin{dfn}
Let $\Delta^k \equiv \{p_V > 0\}$ be the strictly positive probability 
simplex of dimension $k = 2^{|V|}-1$.  We say that a homeomorphism 
$\theta : \Delta^k \rightarrow \Theta \subseteq \reals^k$ onto an open
set $\Theta$ is a \emph{smooth parameterization} of $\Delta^k$ if 
$\theta$ is twice continuously differentiable, and its Jacobian
has full rank $k$ everywhere.
\end{dfn}

The canonical smooth parameterization of $\Delta^k$ is via
\emph{log-linear parameters} $\eta_L$, defined by the M\"obius
expansion
\begin{align*}
\log p_V(x_V) = \sum_{L \subseteq V} (-1)^{|x_L|} \eta_L;
\end{align*}
here $|x_L| = \sum_{v \in L} x_v$ is the number of 1s in $x_L \in
\{0,1\}^L$.  It follows by M\"obius inversion that
\begin{align}
  \eta_L = 2^{-|V|} \sum_{x_V \in \{0,1\}^{|V|}} (-1)^{|x_L|} \log
  p_V(x_V); \label{eqn:llp}
\end{align}
see, for example, \citet{lau:96}.
For example, if $V = \{1,2,3\}$,
\begin{align*}
\eta_{13} = \frac{1}{8} \log \frac{p(0, 0, 0) \, p(0, 1, 0) \, p(1, 0, 1) \, p(1, 1, 1)}{p(1, 0, 0) \, p(1, 1, 0) \, p(0, 0, 1) \, p(0, 1, 1)}.
\end{align*}
It is well known that the collection $\bs\eta \equiv (\eta_L, \,
\emptyset \neq L \subseteq V)$ provides a smooth parameterization of
the joint distribution $p_V$ with $\Theta = \reals^k$.

\begin{dfn}
We define a \emph{marginal log-linear parameter} by analogy with
(\ref{eqn:llp}), as the ordinary log-linear parameter for a particular
marginal distribution.  Let $L \subseteq M \subseteq V$; then
\begin{align*}
\lambda_L^M &= 2^{-|M|} \!\!\!\! \sum_{x_{\sn M} \in \{0,1\}^{|M|}} (-1)^{|x_L|} \log p_M(x_M)
\end{align*}
is the marginal log-linear parameter associated with the \emph{margin}
$M$ and the \emph{effect} $L$.  See \citet{br:02}.
\end{dfn}

Clearly $\lambda_L^V = \eta_L$ and, for example,
\begin{align*}
\lambda_{13}^{13} &= \frac{1}{4} \log \frac{p_{13}(0,0) \, p_{13}(1,1)}{p_{13}(1,0) \, p_{13}(0,1)},
\end{align*}
which is the log-odds ratio between $X_1$ and $X_3$.  In order to fit
a model with the constraint $X_1 \indep X_3$ we could choose a
parameterization that includes $\lambda_{13}^{13}$, and fix it to be
zero.

One way to characterize the main idea of \citet{br:02} is as follows:
given some arbitrary margins $p_{M_1}, \ldots, p_{M_k}$ of a joint
distribution $p_V$, what additional information does one need to
smoothly reconstruct the full joint distribution $p_V$?  They show
that one possibility is to take the collection of log-linear
parameters $\eta_L = \lambda_L^V$ where $L \nsubseteq M_i$ for any $i
= 1,\ldots,k$.

It follows that given any inclusion-respecting sequence of margins
$M_1, \ldots, M_k =V$ (i.e.\ $M_i \subseteq M_j$ only if $i < j$), we
can smoothly parameterize $p_V$ with marginal log-linear parameters of
the form $\lambda_L^{M_i}$, where $L \subseteq M_i$ but $L \nsubseteq
M_j$ for any $j < i$.  

\begin{exm} \label{exm:mllp} Take the inclusion-respecting sequence of
  margins $\{1,2\}$, $\{2,3\}$, $\{1,2,3\}$.  This gives us the smooth
  parameterization consisting of the vector $\bs\lambda_{\mathcal{P}}$
  below.  The pairs $(L,M)$ are summarized (grouped by margin) in the
  adjacent table.\footnote{Note that here, and in the sequel, we
    abbreviate sets of integers by omitting the braces and commas in
    order to avoid overburdened notation: so, for example, $23$ means
    $\{2,3\}$.}
\begin{align*}
&\mathcal{P}: \bmllp
12&1, 2, 12\\
23&3, 23\\
123&13, 123.\\
\emllp
&& \bs\lambda_{\mathcal{P}} = (\lambda_1^{12}, \lambda_{2}^{12}, \lambda_{12}^{12}, 
\lambda_{3}^{23}, \lambda_{23}^{23}, \lambda_{13}^{123}, \lambda_{123}^{123})^T.
\end{align*}
\end{exm}

Now, let $\mathcal{P}$ be an \textbf{arbitrary} collection of
effect-margin pairs $(L,M)$ such that $\emptyset \neq L \subseteq M
\subseteq V$.  Define
\[
\bs\lambda_{\mathcal{P}} = \bs\lambda_{\mathcal{P}}(\bs p) =
(\lambda_L^M : (L,M) \in \mathcal{P})
\]
to be the corresponding vector of marginal log-linear parameters.  The
main question considered by this paper is: under what circumstances
does $\bs\lambda_{\mathcal{P}}$ constitute a smooth parameterization
of $p_V$?

\subsection{Existing Results}

We say that $\mathcal{P}$ is \emph{complete} if every non-empty subset
of $V$ appears as an effect in $\mathcal{P}$ exactly once.  If, in
addition, the margins can be ordered so that each effect appears with
the first margin of which it is a subset, we say that $\mathcal{P}$ is
\emph{hierarchical}.  Parameterizations that can be constructed from
an inclusion-respecting sequence of margins in the manner of Example
\ref{exm:mllp} correspond precisely to hierarchical $\mathcal{P}$.
\citet{br:02} show that if $\mathcal{P}$ is complete and hierarchical
then $\bs\lambda_{\mathcal{P}}$ gives a smooth parameterization of the
joint distribution; in addition, they show that completeness is
necessary for smoothness.  \citet{forcina:12} shows that if
$\mathcal{P}$ is complete and contains only two distinct margins $M$,
then $\bs\lambda_{\mathcal{P}}$ is smooth.  

To our knowledge, these are the only existing results on the
smoothness of marginal log-linear parameterizations.  No example has
been provided of a complete parameterization which is non-smooth.  In
Sections \ref{sec:smooth}, \ref{sec:fixed} and \ref{sec:cycles} we
will show that, in fact, many more complete parameterizations are
smooth than had previously been known.

The issue of smoothness in non-hierarchical models was raised by
\citet{forcina:12} in the context of loop models of conditional
independence, and expanded upon by \citet{colombi:14} for models of
context-specific conditional independence; the latter consider a more
general class of models than we do, but there is no overlap in the
theoretical results.  Examples of ordinary conditional independences
models that require non-complete parameterizations (and therefore are
not curved exponential families) are found in \citet{drton:09}.

\section{An Analytical Map between Margins} \label{sec:smooth}

To parameterize a marginal distribution $p_M$ we can use the marginal
log-linear parameters $\{\lambda_L^M : \emptyset \neq L \subseteq
M\}$.  An analogous result holds for conditional distributions: for
disjoint $A,B$ define
\[
\bs\lambda_{A|B} \equiv (\lambda_L^{AB} \,|\, L \subseteq A \cup B, L \cap A \neq \emptyset);
\]
in other words, all the MLL parameters for the margin $A \cup B$ whose
effect contains some element of $A$.  Then $\bs\lambda_{A|B}$
constitutes a smooth parameterization of the conditional distribution
$X_A \,|\, X_B$

A consequence of this is to aid us in understanding the relationship
between log-linear parameters defined within different margins.
Theorem 3 of \citet{br:02} shows that distinct MLL parameters
corresponding to the same effect in different margins (i.e.\
$\lambda_L^M$ and $\lambda_L^N$ with $M \neq N$) are linearly
dependent at certain points in the parameter space, and that therefore
no smooth parameterization can include two such parameters.  The
following theorem elucidates the exact relationship between such
parameters, and will later be used to demonstrate the smoothness of
certain non-hierarchical parameterizations.

\begin{thm} \label{thm:linear} Let $A,M$ be disjoint subsets of $V$.
  The log-linear parameter $\lambda_L^{MA}$ may be decomposed as
\begin{align}
\lambda_L^{MA} = \lambda_L^{M} + f(\bs\lambda_{A|M}), \label{eqn:equal}
\end{align}
for a smooth function $f$, which vanishes whenever $X_A \indep X_v
\,|\, X_{M \setminus \{v\}}$ for some $v \in L$.

In addition, if $K \cap (V \setminus M) \neq \emptyset$
\begin{align}
\frac{\partial \lambda_L^{M}}{\partial \eta_K} = 
2^{-|M|} \sum_{x_V} (-1)^{|x_{K \triangle L}|} p(x_{V \setminus M} \,|\, x_M),
\label{eqn:deriv}
\end{align}
(where $(\eta_J : J \neq K)$ are held fixed).
\end{thm}

\begin{proof}
We have
\begin{align*}
\lambda_L^{MA} &= 2^{-|AM|} \sum_{x_{MA}} (-1)^{|x_L|} \log p(x_{MA})\\
  &= 2^{-|AM|} \!\! \sum_{x_{MA}} (-1)^{|x_L|} \left[\log p(x_M) + 
		   \log p(x_{A} \,|\, x_M)\right]\\
  &= 2^{-|M|} \sum_{x_{M}} (-1)^{|x_L|} \log p(x_M) + 
		  2^{-|AM|} \sum_{x_{MA}} (-1)^{|x_L|} \log p(x_{A} \,|\, x_M)\\
  &= \lambda_L^{M} + 
		  2^{-|AM|} \sum_{x_{MA}} (-1)^{|x_L|} \log p(x_{A} \,|\, x_M).
\end{align*}
Since the second term is a smooth function of the conditional
probabilities $p(x_A \,|\, x_M)$, it follows that it is also a smooth
function of the claimed parameters.  The implication of independence
follows from Lemma 2.9 of \citet{evans:13}.

Now,
\begin{align*}
\frac{\partial}{\partial \eta_K} p(x_V) &= \frac{\partial}{\partial \eta_K} \exp\left(\sum_{J \subseteq V} (-1)^{|x_J|} \eta_J \right)
= (-1)^{|x_K|} p(x_V),
\end{align*}
and similarly 
\begin{align*}
\frac{\partial}{\partial \eta_K} p(x_M) &= \frac{\partial}{\partial \eta_K} \sum_{y_{V \setminus M}} p(y_{V\setminus M}, x_M) = (-1)^{|x_{K \cap M}|} \sum_{y_{V \setminus M}} (-1)^{|y_{K \setminus M}|} p(y_{V\setminus M}, x_M).
\end{align*}
Hence the derivative of (\ref{eqn:equal}) in the case $A = V \setminus
M$ becomes
\begin{align*}
\frac{\partial f}{\partial \eta_K} &= 
2^{-|V|} \sum_{x_{V}} (-1)^{|x_L|} \left\{ (-1)^{|x_K|} \frac{p(x_V)}{p(x_V)} - \frac{(-1)^{|x_{K \cap M}|}}{p(x_M)} \sum_{y_{V \setminus M}} (-1)^{|y_{K \setminus M}|} p(y_{V\setminus M}, x_M) \right\}\\
&= 0 - 2^{-|V|} \sum_{x_{V}} (-1)^{|x_{L}|+|x_{K \cap M}|}  \sum_{y_{V \setminus M}} (-1)^{|y_{K \setminus M}|} p(y_{V \setminus M} \,|\, x_M)
\intertext{and, since there is no dependence upon $x_{V \setminus M}$, this is the same as}
&= - 2^{-|V|} 2^{|V \setminus M|} \sum_{x_{V}} (-1)^{|x_{L}|+|x_{K \cap M}|+|x_{K \setminus M}|} p(x_{V \setminus M} \,|\, x_M).
\end{align*}
Then note that $|x_{L}|+|x_{K \cap M}|+|x_{K \setminus M}| = |x_L| +
|x_K|$ simply counts the number of 1s in $L$ and in $K$, so $|x_{L
  \triangle K}|$ is even if and only if $|x_L| + |x_K|$ is.  Hence
\begin{align*}
\frac{\partial f}{\partial \eta_K} 
&= - 2^{-|M|} \sum_{x_{V}} (-1)^{|x_{L \triangle K}|} p(x_{V \setminus M} \,|\, x_M),
\end{align*}
which gives the required result.
\end{proof}

\begin{rmk}
  We have shown that if the conditional distribution of $X_A$ given
  $X_M$ is fixed the relationship between $\lambda_L^M$ and
  $\lambda_L^{MA}$ (and indeed any parameter of the form
  $\lambda_L^{MB}$ for $B \subseteq A$) is linear.  In particular, if
  we \emph{know} $p_{A|M}$, then $\lambda_L^{MA}$ and $\lambda_L^{M}$
  become interchangeable as part of a parameterization, preserving
  smoothness and (when relevant) variation independence.
\end{rmk}

\subsection{Constructing Smooth Parameterizations}

The following example shows how Theorem \ref{thm:linear} can be used
to prove the smoothness of a parameterization.

\begin{exm}
Consider the complete collections $\mathcal{P}$ and $\mathcal{Q}$ below.
\begin{align*}
\mathcal{P}:&\bmllp
3&3\\
23&23\\
123&1, 2, 12, 13, 123.\\
\emllp && 
\mathcal{Q}:
\bmllp
3&3\\
23&2, 23\\
123&1, 12, 13, 123\\
\emllp.
\end{align*}
$\mathcal{P}$ is not hierarchical because in any inclusion-respecting
ordering the margin 23 must precede 123, in which case the effect 2
(contained in the pair $(2,123)$) is not associated with the first
margin of which it is a subset.  Existing results therefore cannot
tell us whether or not $\bs\lambda_{\mathcal{P}}$ is smooth.  However,
by fixing the parameters $\bs\lambda_{1|23} = (\lambda_1^{123},
\lambda_{12}^{123}, \lambda_{13}^{123}, \lambda_{123}^{123})$ Theorem
\ref{thm:linear} shows that $\lambda_{2}^{123}$ and $\lambda_{2}^{23}$
are interchangeable.  Hence $\bs\lambda_{\mathcal{P}}$ is smooth if
and only if $\bs\lambda_{\mathcal{Q}}$ is also smooth which, since
$\mathcal{Q}$ satisfies the conditions of a hierarchical
parameterization, it is.  In addition, $\bs\lambda_{\mathcal{P}}$ and
$\bs\lambda_{\mathcal{Q}}$ are both variation independent
parameterizations (i.e.\ any $\bs\lambda_{\mathcal{P}} \in \reals^7$
corresponds to a valid probability distribution).
\end{exm}

We generalize the approach used in the preceding example with the
following definition and proposition.

\begin{dfn}
Let $\mathcal{P}$ be a collection of MLL parameters, and define 
\[
\mathcal{P}_{-v} = \{(L, M \setminus \{v\}) \,|\, (L, M) \in \mathcal{P}, \, v \notin L \}.
\]
That is, all effects involving $v$ are removed, and any margins $M$
containing $v$ are replaced by $M \setminus \{v\}$.
\end{dfn}

\begin{prop} \label{prop:reduce} Let $\mathcal{P}$ be a complete
  collection of marginal log-linear parameters over $V$ such that the
  variable $v$ is not in any margin except $V$.  Then
  $\bs\lambda_{\mathcal{P}}$ is a smooth parameterization of $X_V$ if and
  only if $\bs\lambda_{\mathcal{P}_{-v}}$ is a smooth parameterization of
  $X_{V \setminus v}$.  In addition, $\bs\lambda_{\mathcal{P}}$ is
  variation independent if and only if $\bs\lambda_{\mathcal{P}_{-v}}$ is.
\end{prop}

\begin{proof}
  Since $V$ is the only margin containing $v$ and the parameterization
  is complete, we have the parameters $\bs\lambda_{v|V\setminus v} =
  (\lambda_{L}^V : v \in L)$.  Hence we can smoothly parameterize the
  distribution of $X_v \,|\, X_{V \setminus v}$ with these parameters.

  By Theorem \ref{thm:linear}, any other parameter $\lambda_L^V$ such
  that $v \notin L$ is (having fixed the distribution of $X_v \,|\,
  X_{V \setminus v}$) a smooth function of $\lambda_L^{V \setminus
    v}$.  It follows that we have a smooth map between
  $\bs\lambda_{\mathcal{P}}$ and $(\bs\lambda_{\mathcal{P}_{-v}},
  \bs\lambda_{v|V\setminus v})$.  Since
  $\bs\lambda_{\mathcal{P}_{-v}}$ is a function of $p_{V\setminus v}$,
  and $\bs\lambda_{v|V\setminus v}$ smoothly parameterizes
  $p_{v|V\setminus v}$, it follows that $\bs\lambda_{\mathcal{P}}$
  smoothly parameterizes $p_V$ if and only if
  $\bs\lambda_{\mathcal{P}_{-v}}$ smoothly parameterizes $p_{V
    \setminus v}$.

  Lastly, the two pieces $\bs\lambda_{\mathcal{P}_{-v}}$ and
  $\bs\lambda_{v|V\setminus v}$ are variation independent of one
  another because this is a parameter cut, and parameters within
  $\bs\lambda_{v|V\setminus v}$ are all variation independent since
  they are just ordinary log-linear parameters; therefore
  $\bs\lambda_{\mathcal{P}_{-v}}$ is variation independent if and only
  if $\bs\lambda_{\mathcal{P}}$ is.
\end{proof}

\begin{cor} \label{cor:nested} Any complete parameterization in which
  the margins are strictly nested ($M_1 \subset M_2 \subset \cdots
  \subset M_k = V$) is smooth and variation independent.
\end{cor}

Lemma 6 of \citet{forcina:12} deals with the special case $k=2$, which
to our knowledge was the only prior result showing that a
non-hierarchical MLL parameterization may be smooth.

\begin{exm}
Consider
\begin{align*}
&\mathcal{P}: \bmllp
13&3\\
23&23\\
123&1, 2, 12, 13, 123\\
\emllp && 
\mathcal{Q}: \bmllp
13&3\\
123&1, 2, 12, 23, 13, 123\\
\emllp.
\end{align*}
$\mathcal{P}$ does not satisfy the conditions of Proposition
\ref{prop:reduce}; however, applying Theorem \ref{thm:linear} shows
that $\lambda_{23}^{23}$ is just a linear function of
$\lambda_{23}^{123}$ after fixing the other parameters in the margin
$123$, so $\mathcal{P}$ is smooth if and only if $\mathcal{Q}$ is.
Applying Corollary \ref{cor:nested} shows that
$\bs\lambda_{\mathcal{Q}}$ (and therefore $\bs\lambda_{\mathcal{P}}$) is smooth.
\end{exm}

\begin{prop} \label{prop:cond} 
  Let $\mathcal{P}$ be a complete parameterization, and suppose that
  for some $v \in V$, and every $A \subseteq V \setminus \{v\}$, the
  sets $A \cup \{v\}$ and $A$ appear as effects within the same margin
  in $\mathcal{P}$.

  Then $\bs\lambda_{\mathcal{P}}$ is a smooth parameterization of $X_V$
  if and only if $\bs\lambda_{\mathcal{P}_{-v}}$ is a smooth
  parameterization of $X_{V \setminus v}$.  In addition,
  $\bs\lambda_{\mathcal{P}}$ is variation independent if and only if
  $\bs\lambda_{\mathcal{P}_{-v}}$ is variation independent.
\end{prop}

\begin{proof}
  Since $A \subseteq V \setminus \{v\}$ and $A \cup \{v\}$
  appear in the same margin, say $M \cup \{v\}$, set
\begin{align*}
\kappa_A^M(x_v) &= \lambda_A^{Mv} + (-1)^{|x_v|} \lambda_{Av}^{Mv}\\
  &= 2^{-|Mv|} \sum_{y_{Mv}} (-1)^{|y_A|} \left[(-1)^{|y_A|} + (-1)^{|y_{Av}| + |x_v|}\right] \log p_{M|v}(y_M, y_v)
\intertext{which is zero unless $x_v = y_v$, leaving}
  &= 2^{-|M|} \sum_{y_{M}} (-1)^{|y_A|} \log p_{M|v}(y_M, x_v)\\
  &= 2^{-|M|} \sum_{y_{M}} (-1)^{|y_A|} \log p_{M|v}(y_M \,|\, x_v).
\end{align*}
%Then, as in the proof of Theorem \ref{thm:param}, 
But notice this is of the same form as an MLL parameter for the pair
$(A, M)$ over the conditional distribution $p_{V\setminus v|v}(\cdot
\,|\, x_v)$.  It follows that for fixed $x_v$ the parameters $\left\{
  \kappa_A^M(x_v): (A,M) \in \mathcal{P}, v \notin A \right\}$ form a
complete MLL collection of the form $\mathcal{P}_{-v}$ for the
conditional distribution of $X_{V\setminus v} \,|\, X_v = x_v$.  If
$\bs\lambda_{\mathcal{P}_{-v}}$ is smooth then we can smoothly recover
the conditional distribution $p_{V \setminus v|v}$.  Furthermore, if
the effect $\{v\}$ is in a margin $N\cup \{v\}$, then using
(\ref{eqn:equal}) we obtain
\[
\lambda_v^{Nv} = \lambda_v^v + f(p_{N|v}),
\]
and smoothly recover $\lambda_v^v$.  In addition $\lambda_v^v$ is
variation independent of $p_{N|v}$ (since $p_v$, $p_{N|v}$ constitutes
a parameter cut) and has range $\reals$, so the same is true of
$\lambda_v^{Nv}$.

Conversely if $\bs\lambda_{\mathcal{P}}$ is smooth, then given
parameters $\bs\lambda_{\mathcal{P}_{-v}}$ we can set up a dummy distribution
on $p_V$ in which $\kappa_A^M(x_v)=\lambda_A^M$ for each $x_v$, and
$\lambda_v^N = 0$, thus smoothly recovering $p_{V \setminus v}$.
\end{proof}

\begin{exm} \label{exm:cond}
As an example, consider 
\begin{align*}
\mathcal{P}: \bmllp
12&2, 12\\
13&3, 13\\
123&1, 23, 123\\
\emllp
\end{align*}
which is not hierarchical, and nor does it satisfy the conditions of
Proposition \ref{prop:reduce}.  However it does satisfy the the
conditions of Proposition \ref{prop:cond} for $v=1$, and
$\mathcal{P}_{-1}=\{(2,2), (3,3), (23,23)\}$,
%\begin{center}
%\bmllp
%2&2\\
%3&3\\
%23&23,\\
%\emllp
%\end{center}
which is hierarchical and so certainly smooth.  Hence $\mathcal{P}$ 
represents a smooth parameterization.
\end{exm}

% \begin{exm}
% Consider the model defined by the conditional independences
% \begin{align*}
% X_1 \indep X_2 \,|\, X_4 \qquad X_1 \indep X_5 \,|\, X_2 \qquad X_1 \indep X_3 \,|\, X_5 \qquad X_2 \indep X_3 \,|\, X_1.
% \end{align*}
% It is not clear from existing results whether this model is a curved
% exponential family or not.  The model can be fixed by setting the
% parameters in $\mathcal{P}$ (below) to zero \citep[][Lemma 1]{rudas:10}.
% \begin{align*}
% \mathcal{P}: \bmllp
% 124&12, 124\\
% 125&15, 125\\
% 135&13, 135\\
% 123&23, 123\\
% \emllp
% & &\mathcal{Q}:
% \bmllp
% 124&2, 12, 4, 14, 24, 124\\
% 125&5, 15, 25, 125\\
% 135&3, 13, 35, 135\\
% 123&23, 123\\
% 12345&(other subsets)\\
% \emllp
% \end{align*}
% One way to prove that the model is a curved exponential family is to
% embed the parameters $\bs\lambda_{\mathcal{P}}$ within a complete
% parameterization which is known to be smooth \citep[][Theorem
% 5]{br:02}.  In this case there is  no hierarchical parameterization
% containing these parameters, but we can embed them into $\mathcal{Q}$
% instead.  This satisfies the conditions of Proposition \ref{prop:cond}
% with $v=1$, so it reduces to
% \begin{align*}
% \mathcal{Q}_{-1}:
% \bmllp
% 24&2, 4, 24\\
% 25&5, 25\\
% 35&3, 35\\
% 23&23\\
% 2345&(other subsets)\\
% \emllp
% \end{align*}
% which \emph{is} hierarchical.  Hence $\bs\lambda_{\mathcal{Q}}$ is
% smooth, and the model defined by $\bs\lambda_{\mathcal{P}} = \bs 0$ is
% a curved exponential family.
% \end{exm}

\section{Fixed Point Mappings} \label{sec:fixed}

The previous section gives analytical maps between some
parameterizations, but Propositions \ref{prop:reduce} and
\ref{prop:cond} only apply directly to a relatively small number of
cases.  In this section we build on these results by presenting
conditions for the existence of a smooth map, even without a closed
form expression.

Given a particular complete MLL parameterization $\bs
\lambda_{\mathcal{P}}$, the identity (\ref{eqn:equal}) in Theorem
\ref{thm:linear} can be written in vector form as 
\begin{align*} 
\bs\eta = \bs\lambda + \bs f(\bs \eta).  
\end{align*}
For a given $\bs\lambda$ this suggests that $\bs\eta$ might be
recovered using fixed point methods; the identity (\ref{eqn:deriv})
gives us information about the Jacobian of $\bs f$.

\begin{exm} \label{exm:psi}
Consider the parameterization based on
\begin{align*}
\mathcal{P}: \bmllp
23&2, 23\\
13&1\\
123&12, 3, 13, 123.\\
\emllp
\end{align*}
If we can smoothly recover $\eta_1$, $\eta_2$ and $\eta_{23}$ from
$\bs\lambda_{\mathcal{P}}$ then it follows that
$\bs\lambda_{\mathcal{P}}$ is a smooth parameterization.  From
(\ref{eqn:equal}) we have
\begin{align*}
\left(\begin{matrix}\eta_2\\\eta_{23} \end{matrix}\right) 
&= \left(\begin{matrix}\lambda_2^{23}\\\lambda_{23}^{23} \end{matrix}\right) + \bs f(\eta_{1},\eta_{12},\eta_{13},\eta_{123});\\
\intertext{since $\eta_{12}$, $\eta_{13}$ and $\eta_{123}$ are given in the parameterization 
we can assume these to be fixed, so abusing notation slightly}
\left(\begin{matrix}\eta_2\\\eta_{23} \end{matrix}\right) &=  \left(\begin{matrix}\lambda_2^{23}\\\lambda_{23}^{23} \end{matrix}\right) + \bs f(\eta_1) = \left(\begin{matrix}\lambda_2^{23} + f_1(\eta_1)\\\lambda_{23}^{23}  + f_2(\eta_1) \end{matrix}\right).
\end{align*}
Similarly,
$\eta_{1} = \lambda_1^{13} + g(\eta_{2}, \eta_{23})$ for some smooth $g$, so
$\eta_1$ is a solution to the equation
\begin{align*}
x &= \lambda_1^{13} + g(\lambda_2^{23} + f_1(x), \lambda_{23}^{23} + f_2(x))\\
&\equiv \Psi(x).
\end{align*}
If $\Psi$ can be shown to be a contraction mapping, then we are
guaranteed to find a unique solution, and therefore recover the joint
distribution.  In addition, if $\Psi$ is a contraction for all
$\bs\eta$, then since it varies smoothly in $\bs\eta$ we will have
shown that $\bs\lambda_{\mathcal{P}}$ is a smooth parameterization.
%We return to this in Example \ref{exm:psi2}.
\end{exm}

% \begin{lem} \label{lem:deriv}
% If $L' \nsubseteq M$, then
% \begin{align*}
% \frac{\partial \lambda_L^{M}}{\partial \lambda_{L'}^{V}} &= -2^{-|M|} \sum_{x_V} (-1)^{|x_{L\triangle L'}|} p(x_{V\setminus M} \,|\, x_M),
% \end{align*}
% (where $\lambda_{K}^{V}$ for $K \neq L'$ are all held fixed) so in
% particular if $p$ is positive then $|\frac{\partial
%   \lambda_L^{M}}{\partial \lambda_{L'}^{V}}| < 1$.
% \end{lem}

% \begin{rmk}
%   Note that if $L=L'$ the right hand side is $-1$, while the true
%   derivative is $+1$.  If $L' \subseteq M$ but $L \neq L'$, then the
% derivative is zero.
% \end{rmk}

Define $\epsilon = \min_{x_V} p(x_V)$ to be the smallest amount of
probability assigned to any cell in our joint distribution, and
$\Delta_{\epsilon} = \{p : \min_{x_V} p(x_V) \geq \epsilon \}$ to
be the probability simplex consisting of such distributions.  The
Jacobian of an otherwise smooth parameterization can become singular on
the boundary of the probability simplex, so it is useful to have
control over this quantity.

The next result allows us to control the magnitude of the columns (or
rows) of the Jacobian of $\Psi$ in certain examples.  The proof is
given in the appendix.

\begin{lem} \label{lem:mag}
Let $J \subseteq M$, and $\emptyset \neq K \subseteq V \setminus M$.  Then
\begin{align*}
\sum_{\emptyset \neq C \subseteq M} \left| \frac{\partial \lambda_{C}^{M}}{\partial \eta_{JK}} \right|^2 \leq 1 - \epsilon.
\end{align*}
Alternatively, if $\emptyset \neq C \subseteq M$, then
\begin{align*}
&\sum_{\substack{J \subseteq M}} \left| \frac{\partial \lambda_C^{M}}{\partial \eta_{JK}} \right|^2 \leq 1 - \epsilon.
\end{align*}
\end{lem}

\begin{exm} \label{exm:psi2} Returning to the parameterization in
  Example \ref{exm:psi}, the derivative of $\Psi$ is
\begin{align*}
\Psi'(x) % &= \nabla g \cdot \frac{\partial \bs f}{\partial x}\\
&= \frac{\partial \lambda_1^{13}}{\partial \eta_{2}} \frac{\partial \lambda_2^{23}}{\partial \eta_{1}} 
+ \frac{\partial \lambda_1^{13}}{\partial \eta_{23}} \frac{\partial \lambda_{23}^{23}}{\partial \eta_{1}},
\end{align*}
which is the dot product of the vectors
\begin{align*}
&\left( \frac{\partial \lambda_1^{13}}{\partial \eta_{2}}, \frac{\partial \lambda_1^{13}}{\partial \eta_{23}} \right)^T && 
\left( \frac{\partial \lambda_2^{23}}{\partial \eta_{1}}, \frac{\partial \lambda_{23}^{23}}{\partial \eta_{1}} \right)^T.
\end{align*}
By applying the two parts of Lemma \ref{lem:mag}, these vectors each
have magnitude at most $1 - \epsilon$.  Hence $|\Psi'(x)| \leq
1-\epsilon$, and $\Psi$ is a contraction on $\Delta_\epsilon$ for
every $\epsilon > 0$.  It follows that the equation has a unique
solution among all positive probability distributions (and this can be
found by iteratively applying $\Psi$ to any initial distribution), and
by the inverse function theorem it is a smooth function of $\bs
\lambda$.  Hence $\bs\lambda_{\mathcal{P}}$ is indeed smooth.
\end{exm}

\begin{rmk}
  \citet{forcina:12} also uses fixed point methods to recover
  distributions from marginal log-linear parameters, but that approach
  involves computing probabilities directly.  We discuss those methods
  in Section \ref{sec:cycles}.
\end{rmk}

Lemma \ref{lem:mag} enables us to formulate the following
generalization of the idea used in the example above.

\begin{lem} \label{lem:fixedpoint}
  Let $\mathcal{P}$ be complete and such that for any $(L,M) \in
  \mathcal{P}$ with $M \subset V$, there is at most one other
  margin $N \subset V$ in $\mathcal{P}$ with $L \cap (V \setminus
  N) \neq \emptyset$.  Then $\bs\lambda_{\mathcal{P}}$ is smooth.
\end{lem}

\begin{proof}
By Theorem \ref{thm:linear},
\begin{align*}
\eta_{L}  = \lambda_{L}^{M} + f(\bs\lambda_{V\setminus M|M}).
\end{align*}
Since $N$ is the only margin in $\mathcal{P}$ such that $L \cap (V
\setminus N) \neq \emptyset$, it follows that all the parameters in
$\bs\lambda_{V \setminus M|M}$ are known and fixed except for
$(\lambda_{K}^V : K \in \mathbb{L}_N)$, where $\mathbb{L}_N$ is the
set of effects contained in the margin $N$.  Hence
\begin{align}
\eta_{L} = \lambda_{L}^{M}  + f(\bs\eta_{\mathbb{L}_N}). \label{eqn:fp}
\end{align}
Now, consider the vector equation obtained by stacking (\ref{eqn:fp})
over all pairs $(L,M) \in \mathcal{P}$.  This defines a fixed point
equation whose solution is $\bs\eta$, and the column of the Jacobian
corresponding to $L$ has non-zero entries
\begin{align*}
&\frac{\partial \lambda_{C}^{N}}{\partial \eta_{L}}, \qquad C \in
\mathbb{L}_N.  
\end{align*} 
From Lemma \ref{lem:mag}, each column has magnitude at most
$1-\epsilon$, and therefore the mapping is a contraction on
$\Delta_{\epsilon}$ for each $\epsilon > 0$.  It follows that the
fixed point equation has a unique solution which, by the inverse
function theorem, is a smooth function of $\bs\lambda$.
\end{proof}

From this result we obtain the following corollary, the conditions of
which are easy to verify.

\begin{cor}
Any complete parameterization with at most three margins is smooth.
\end{cor}

\begin{proof}
  Since one of the margins must be $V$, it is clear that the
  conditions of Lemma \ref{lem:fixedpoint} hold.
\end{proof}

\begin{exm} \label{exm:reduce}
Consider $\mathcal{P}$ below.
\begin{align*}
\mathcal{P}: 
\bmllp
1&1\\
12&2\\
13&3\\
123&12, 13, 23, 123\\
\emllp
&&
\mathcal{Q}: 
\bmllp
1&1\\
123&2, 12, 3, 13, 23, 123\\
\emllp
\end{align*}
Although it does not satisfy the conditions of Lemma
\ref{lem:fixedpoint} directly, one can use the basic idea to set up a
smooth contraction mapping from $\bs\lambda_{\mathcal{P}}$ to
$\bs\lambda_{\mathcal{Q}}$; since $\mathcal{Q}$ is hierarchical, both
parameterizations are smooth.
\end{exm}

\section{Cyclic Parameterizations} \label{sec:cycles}

This section takes a third approach to determining smoothness, by
using Markov chain theory to recover certain marginal distributions.
This method allows us to demonstrate the smoothness of certain
conditional independence models.

\citet[][Example 2]{forcina:12} considers the model defined (up to
some relabelling) by the conditional independences
\begin{align}
&X_1 \indep X_2 \,|\,  X_3, &&X_1 \indep X_3 \,|\, X_4, &&X_1 \indep X_4 \,|\, X_2, \label{eqn:cim}
\end{align}
which is equivalent to setting the parameters
\begin{align}
\bmllp
123&12, 123\\
134&13, 134\\
124&14, 124\\
\emllp \label{eqn:cimp}
\end{align}
to zero.  Note that we cannot embed these parameters into a larger
hierarchical parameterization, because each pairwise effect will
`belong' to a margin preceding it; for example, $12$ is a subset of
$124$, so for hierarchy the margin $123$ must precede $124$; by a
similar argument, $124$ must precede $134$ which must precede $123$.
We therefore have a cyclic parameterization, referred to as a `loop'
by Forcina.  None of the methods used in the previous sections seem
well suited to dealing with this situation.

\citet{forcina:12} presents an algorithm for recovering joint
distributions given parameterizations of this kind, together with a
condition under which it is guaranteed to converge to the unique
solution.  However, this condition is on the spectral radius of a
complicated Jacobian, and is difficult to verify except in a few
special cases: a numerical test is suggested, but this does not
constitute a proof of smoothness.  Here we show that, at least in some
cases, Forcina's algorithm can be recast as a Markov chain whose
stationary distribution is some margin of the relevant probability 
distribution.

\begin{thm} \label{thm:markov} 
  Let $A_1, \ldots, A_k$ be a disjoint sequence of sets with $k \geq
  2$ such that the conditional distributions $p(x_{A_i} \,|\,
  x_{A_{i-1}}) > 0$ for $i = 2,\ldots,k$ are known, together with
  $p(x_{A_1} \,|\, x_{A_k})$.  Then the marginal distributions $p(x_{A_i})$
  are smoothly recoverable.
\end{thm}

\begin{proof}
Define a $|X_{A_1}| \times |X_{A_1}|$ matrix $M$ with entries
\begin{align*}
M(x_{A_1}', x_{A_1}) = \sum_{x_{A_k}} \cdots \sum_{x_{A_2}} p(x_{A_1} \,|\, x_{A_k}) p(x_{A_k} \,|\, x_{A_{k-1}}) \cdots p(x_{A_2} \,|\, x_{A_1}').
\end{align*}
This is a (right) stochastic matrix with strictly positive entries,
and the marginal distribution $p(x_{A_1})$ satisfies
\begin{align*}
p(x_{A_1}) = \sum_{x_{A_1}'} p(x_{A_1}') M(x_{A_1}', x_{A_1}).
\end{align*}
In other words, $p(x_{A_1})$ is an invariant distribution for the
Markov chain with transition matrix defined by $M$.  Since $M$ has a
finite state-space and all transition probabilities are positive, the
chain is positive recurrent and the equations have a unique solution
\citep[see, e.g.][]{norris:97}.  Hence $p(x_{A_1})$ is defined by the
kernel of the matrix $I - M^T$, and this is a smooth function of the
original conditional probabilities.
\end{proof}

\begin{rmk}
  The Markov chain corresponding to $M$ is that which would be
  obtained by picking some $X_{A_1}$, and then evolving $X_{A_i}$
  using $p(x_{A_i} \,|\, x_{A_{i-1}})$ until we get back to $i=1$.  The
  equations can be solved iteratively by repeatedly right
  multiplying any positive vector by $M$, so that it converges to the 
  stationary distribution of the chain; this corresponds precisely to
  Forcina's algorithm.
%
%  There are some obvious extensions of this result: if we know $p(A_i
%  \,|\, A_{i-1}, B)$ for each $i$, then we can recover $p(A_1 \,|\,
% B)$ in the same way.
\end{rmk}

\begin{exm}[\citet{forcina:12}, Example 9] \label{exm:cycle}
Consider the cyclic parameterization $\mathcal{P}$.
\begin{align*}
&\mathcal{P}: \bmllp
12&1, 12\\
23&2, 23\\
13&3, 13\\
123&123\\
\emllp &&
\mathcal{Q}: \bmllp
3&3\\
23&2, 23\\
12&1, 12\\
13&13\\
123&123\\
\emllp
\end{align*}
The parameters corresponding to the first
three margins in $\mathcal{P}$ are equivalent to the conditional
distributions $p_{1|2}$, $p_{2|3}$ and $p_{3|1}$.  Using the
conditionals in the manner suggested by Theorem \ref{thm:markov}, we
can smoothly recover (for example) the margin $p_3$ (or equivalently
$\lambda_3^3$), and consequently $\mathcal{P}$ is equivalent to the
hierarchical parameterization $\mathcal{Q}$.
\end{exm}

\begin{exm} \label{exm:cycle2}
  The parameters (\ref{eqn:cimp}) can be embedded in the
  complete parameterization $\mathcal{P}$ below.
\begin{align*}
&\mathcal{P}: 
\bmllp
123&2, 23, 12, 123\\
134&3, 34, 13, 134\\
124&4, 24, 14, 124\\
1234&(other subsets)
\emllp,&&
\mathcal{P}_{-1}: 
\bmllp
23&2, 23\\
34&3, 34\\
24&4, 24\\
234&234\\
\emllp.
\end{align*}
$\mathcal{P}$ satisfies Proposition \ref{prop:cond} with $v=1$ and
reduces to $\mathcal{P}_{-1}$, which is isomorphic to the smooth
parameterization in Example \ref{exm:cycle}.  Hence $\mathcal{P}$ is 
smooth, and the conditional independence model (\ref{eqn:cim}) is a 
curved exponential family.
\end{exm}

\begin{exm} \label{exm:cycle3} 
Consider the model defined by
\[
X_1 \indep X_2 \,|\, X_3, \qquad X_2 \indep X_4 \,|\, X_1, \qquad X_1 \indep X_3 \,|\, X_4, \qquad X_3 \indep X_4 \,|\, X_2;
\]
it consists of setting the parameters in $\mathcal{P}$ below to zero.
\begin{align*}
& \mathcal{P}: 
\bmllp
123&12, 123\\
124&24, 124\\
134&13, 134\\
234&34, 234\\
\emllp && \mathcal{Q}: 
\bmllp
14&1, 4, 14\\
23&2, 3, 23\\
123&12, 123\\
124&24, 124\\
134&13, 134\\
234&34, 234\\
1234&1234\\
\emllp
\end{align*}
We can embed $\mathcal{P}$ in the complete parameterization
$\mathcal{Q}$.  Note that using $\lambda_4^{14}, \lambda_{14}^{14}$
and the fact that $X_4 \indep X_2 \,|\, X_1$, means we can construct
the conditional distribution $p(x_4 \,|\, x_1, x_2)$.  Similarly we
have $p(x_3 \,|\, x_2, x_4)$, $p(x_1 \,|\, x_3, x_4)$ and $p(x_2 \,|\,
x_1, x_3)$.  In a manner analogous to the previous example, we can set
up a Markov chain whose stationary distribution is the marginal
$p(x_1, x_2)$ as follows.  First pick $x_1^{(0)},x_2^{(0)}$.  Now, for $i > 0$
\bitem 
\item draw $x_4^{(i)}$ from the distribution $p(x_4 \,|\, x_1^{(i-1)}, x_2^{(i-1)})$;
\item draw $x_3^{(i)}$ from the distribution $p(x_3 \,|\, x_2^{(i-1)}, x_4^{(i)})$;
\item draw $x_1^{(i)}$ from the distribution $p(x_1 \,|\, x_3^{(i)}, x_4^{(i)})$;
\item draw $x_2^{(i)}$ from the distribution $p(x_2 \,|\, x_1^{(i)},
  x_3^{(i)})$.  
\eitem 
Then the distribution of $(x_1^{(i)}, x_2^{(i)})$ converges to $p_{12}$.
We can therefore smoothly recover a distribution satisfying the conditional
independence constraints from the 7 free parameters.  The dimension of the 
model is full, so we have a smooth parameterization of the model, which is
therefore a curved exponential family \citet{lau:96}.
\end{exm}

Note that the construction of the Markov chain in Example
\ref{exm:cycle3} is only possible when the conditional independence
constraints hold, so---unlike in Examples \ref{exm:cycle} and
\ref{exm:cycle2}---we have not actually demonstrated that
$\bs\lambda_{\mathcal{Q}}$ is generally smooth, only that the model
defined by setting $\bs\lambda_{\mathcal{P}} = \bs 0$ is a curved
exponential family.

\begin{rmk}
  Some conditional independence models are non-smooth: e.g.\ the model
  defined by $X_1 \indep X_2, X_4$ and $X_2 \indep X_4 \,|\, X_1, X_3$
  \citep{drton:09}.  This is essentially because it requires that
  $\lambda_{124}^{124} = \lambda_{124}^{1234} = 0$, and setting
  repeated (non-redundant) effects to zero always leads to non-smooth
  parameterizations.

  We remark that all discrete conditional independence models on four
  variables either require repeated effects to be constrained in
  different margins, or can be shown to be smooth using the results of
  this section.  However, the next example shows that for five
  variables the picture is incomplete.
\end{rmk}

\begin{exm} \label{exm:cin}
The conditional independence model defined by 
\begin{align*}
& X_1 \indep X_2 \,|\, X_3, && X_1 \indep X_5 \,|\, X_2, && X_1 \indep X_3 \,|\, X_4, \\
& X_3 \indep X_5 \,|\, X_1, &&  X_3 \indep X_4 \,|\, X_2, X_5
\end{align*}
contains no repeated effects, and yet does not appear to be
approachable using the methods outlined above.  Empirically, Forcina's
algorithm seems to converge to the correct solution, which suggests
that the model is indeed smooth.
\end{exm}

\section{Discussion} \label{sec:discuss}

We have presented three new approaches to demonstrating that complete
but non-hierarchical marginal log-linear parameterizations are smooth,
although a general result eludes us.  Note that each of the approaches
provides an explicit algorithm for obtaining the probabilities from
the parameterization, either using the map in Section
\ref{sec:smooth}, the fixed point iteration in Section
\ref{sec:fixed}, or the Markov chain in Section \ref{sec:cycles}.

There are 104 complete MLL
parameterizations on three variables, of which 23 are hierarchical and
a further 4 consist of only two margins, so are smooth by the results
of \citet{br:02} and \citet{forcina:12} respectively.  These 27 were
the only ones known to be smooth prior to this paper.

A further 5 can be shown smooth using Proposition \ref{prop:reduce},
and one using Proposition \ref{prop:cond} (Example \ref{exm:cond}).
Another 26 can be dealt with using Lemma \ref{lem:fixedpoint} in
combination with other methods, and the approach in Example
\ref{exm:reduce} can be applied to three more.  Example
\ref{exm:cycle} brings the total number of known smooth models to 63.

In addition, of the remaining 41 complete parameterizations, there are
smooth mappings between a group of four and a group of three, so it
remains to establish the smoothness (or otherwise) of at most 36
distinct parameterizations.  As an example of a parameterization whose
smoothness is still not established, consider:
\begin{align*}
& \mathcal{P}: 
\bmllp
12&1, 2\\
13&3, 13\\
23&23\\
123&12, 123\\
\emllp.
\end{align*}
We conjecture that any complete parameterization is smooth, a result
which would enable us to show that models such as that given in
Example \ref{exm:cin} are curved exponential families of
distributions.

\begin{conj}
Any complete MLL parameterization is smooth.
\end{conj}

\bibliographystyle{plainnat}
\bibliography{mybib}

\newpage

\appendix

\section{Technical Proofs}

\subsection{Proof of Lemma \ref{lem:mag}}

\begin{lem} \label{lem:alt}
  Let $\bs d = (d_A)$ be a vector indexed by subsets $A \subseteq
  \{1,\ldots,k\}$.  Then $\|\bs d\| < 1$ if and only if for any 
$B \subseteq [k] \equiv \{1,\ldots,k\}$,
\begin{align*}
\left|\sum_{A \subseteq [k]} (-1)^{|A \cap B|} d_A \right| < 1.
\end{align*}
\end{lem}

\begin{proof}
  The $2^k\times 2^k$-matrix $M$ with $(B,A)$th entry $M_{B,A} =
  2^{-k/2} (-1)^{|A \cap B|}$ is orthogonal, and therefore preserves
  vector lengths.  Then the vector $M \bs d$ has entries with
  magnitude at most $2^{-k/2}$, and therefore has total magnitude at
  most 1.  The same is therefore true of $\bs d$.
\end{proof}

\begin{proof}[Proof of Lemma \ref{lem:mag}]
For $C \subseteq M$, define 
\begin{align*}
d_{C} &\equiv -2^{-|M|} \sum_{x_V} (-1)^{|x_{C \triangle {(J \cup K)}}|} p(x_{V\setminus M} \,|\, x_M)\\
&= -2^{-|M|} \sum_{x_V} (-1)^{|x_{C \triangle J}| + |x_K|} p(x_{V\setminus M} \,|\, x_M)
\end{align*}
so that $d_{C} = \frac{\partial \lambda_{C}^{M}}{\partial
  \eta_{JK}}$ for $C \neq \emptyset$.  Given $y_M$, 
\begin{align*}
\sum_{C \subseteq M} (-1)^{|y_C|} d_{C} &=  -2^{-|M|} \sum_{C \subseteq M} (-1)^{|y_{C}|} \sum_{x_V} (-1)^{|x_{C \triangle J}| + |x_K|} p(x_{V\setminus M} \,|\, x_M),
\end{align*}
and note that
\begin{align*}
\lefteqn{\sum_{C \subseteq M} (-1)^{|y_{C}|} (-1)^{|x_{C \triangle J}|}}\\
&= \sum_{C \subseteq M \setminus \{v\}} (-1)^{|y_{C}|} (-1)^{|x_{C \triangle J}|} + (-1)^{|y_{Cv}|} (-1)^{|x_{Cv \triangle J}|}\\
&= \sum_{C \subseteq M \setminus \{v\}} (-1)^{|y_{C}|} (-1)^{|x_{C \triangle J}|} \left\{1+(-1)^{|y_v|} (-1)^{|x_{v}|} \right\}\\
\intertext{where the expression in braces is 2 if $x_v=y_v$ or 0 otherwise, so}
&= \sum_{C \subseteq M \setminus \{v\}} 2 (-1)^{|y_{C}|} (-1)^{|x_{C \triangle J}|} \Ind_{\{x_v = y_v\}}\\
&= 2^{|M|} \Ind_{\{x_M = y_M\}}.
\end{align*}
Hence
\begin{align*}
\sum_{C \subseteq M} (-1)^{|y_C|} d_{C} &=  - \sum_{x_{V \setminus M}} (-1)^{|x_K|} p(x_{V\setminus M} \,|\, y_M),
\end{align*}
which is an alternating sum of probabilities which sum to one, so
has absolute value at most $1 - \epsilon$.  The result follows from Lemma
\ref{lem:alt}.
The second result is essentially identical, due to the symmetry between
$L,K$ in (\ref{eqn:deriv}).
\end{proof}

\end{document}